\documentclass{article}
\usepackage{amssymb}

\newtheorem{theorem}{Theorem}[section]

\newtheorem{condition}{Condition}[section]

\newtheorem{criterion}{Criterion}[section]
\newtheorem{definition}{Definition}[section]

\newtheorem{proposition}{Proposition}[section]

\newenvironment{proof}[1][Proof]{\noindent\textbf{#1.} }{\ \rule{0.5em}{0.5em}}

\input{tcilatex}
\begin{document}

\title{Automorphic Equivalence in the Varieties of Representations of Lie
algebras.}
\author{A. Tsurkov \\
Mathematical Department, CCET,\\
Federal University of Rio Grande do Norte (UFRN),\\
Av. Senador Salgado Filho, 3000,\\
Campus Universit\'{a}rio, Lagoa Nova, \\
Natal - RN - Brazil - CEP 59078-970 \\
arkady.tsurkov@gmail.com}
\maketitle

\begin{abstract}
In this paper we consider the very wide class of varieties of
representations of Lie algebras over the field $k$, which has characteristic 
$0$. We study the relation between the geometric equivalence and automorphic
equivalence of the representations of these varieties.

If we denote by $\Theta $ one of these varieties, then $\Theta ^{0}$ is a
category of the finite generated free representations of the variety $\Theta 
$. In this paper, we calculate for the considered varieties the quotient
group $\mathfrak{A/Y}$, where $\mathfrak{A}$ is a group of all the
automorphisms of the category $\Theta ^{0}$ and $\mathfrak{Y}$ is a subgroup
of all the inner automorphisms of this category. The quotient group $%
\mathfrak{A/Y}$ measures the difference between the geometric equivalence
and automorphic equivalence of representations of the variety $\Theta $.

In \cite{ShestTsur}, the situation when $\Theta $ is the variety of all the
representations of Lie algebras over the field $k$, which has characteristic 
$0$, was considered. The problem was resolved by the reduction to some
variety of the one-sorted algebraic structures, so, the considerations were
somewhat long and sophisticated. The many-sorted approach to the method of
verbal operations for computation of the quotient group $\mathfrak{A/Y}$ was
elaborated on in \cite{TsurkovManySorted}. By this approach, the result of 
\cite{ShestTsur} was proven again in \cite{TsurkovManySorted} in a simpler
manner. The method of \cite{TsurkovManySorted} allows us to achieve in this
paper a result for the wide class of subvarieties of the variety of all the
representations of Lie algebras.

In other classes of algebraic structures, we have a completely different
situation. In the variety of all the groups, a similar result was achieved
only for the subvariety of all the Abelian groups \cite{PlotkinZhitom} and
for the subvarieties of all the nilpotent groups of the class $\leq d$,
where $d\in 
\mathbb{N}
$, $d\geq 2$ \cite{TsurkovNilpotent}. In the theory of the representations
of groups a similar result was achieved only for the variety of all the
representations of groups in \cite{PlotkinZhitom} and proved again in \cite%
{TsurkovManySorted}.

In Section \ref{example}, we present one example of the subvariety $\Theta $
of the variety of all the representations of the Lie algebras over the field 
$k$, and two representations from the variety $\Theta $ which are
automorphically equivalent but not geometrically equivalent.
\end{abstract}

\section{Introduction\label{intro}}

\setcounter{equation}{0}

A compilation of all the definitions of the basic notions of the universal
algebraic geometry can be found, for example, in \cite{PlotkinVarCat}, \cite%
{PlotkinNotions} and \cite{PlotkinSame}. Also, there are the fundamental
articles \cite{BMR} and \cite{MR}. The natural question of the universal
algebraic geometry is: "When do two universal algebras $H_{1}$ and $H_{2}$
from the some variety $\Theta $ have the same algebraic geometry"? But,
first of all, what does it mean that "two algebras have same algebraic
geometry"? The two notions of geometric equivalence and automorphic
equivalence can give a answer to this question. This article focuses on the
relationship between these notions in a very wide class of the varieties of
the representations of the Lie algebras over the field $k$, which has
characteristic $0$.

We consider the representations of the Lie algebras as two-sorted universal
algebras: the first sort is a sort of elements of Lie algebras, and the
second sort is a sort of vectors of linear spaces. Therefore, we will
consider all basic notions of the universal algebraic geometry in the
many-sorted version as in \cite{TsurkovManySorted}. We suppose that there is
a finite set of names of sorts $\Gamma $. In our case $\Gamma =\left\{
1,2\right\} $. Many-sorted algebra, first of all, is a set $H$ with the
"sorting": mapping $\eta _{H}:H\rightarrow \Gamma $. The set of elements of
the sort $i$, where $i\in \Gamma $, of the algebra $H$ will be the set $\eta
_{H}^{-1}\left( i\right) $. We denote $\eta _{H}^{-1}\left( i\right)
=H^{\left( i\right) }$. If $h\in H^{\left( i\right) }$, then many times we
will denote $h=h^{\left( i\right) }$, with a view to emphasizing that $h$ is
an element of the sort $i$.

We denote by $\Omega $ the signature (set of operations) of our algebras.
Every operation $\omega \in \Omega $ has a type $\tau _{\omega }=\left(
i_{1},\ldots ,i_{n};j\right) $, where $n\in 
\mathbb{N}
$, $i_{1},\ldots ,i_{n},j\in \Gamma $. Operation $\omega \in \Omega $ of the
type $\left( i_{1},\ldots ,i_{n};j\right) $ is a partially defined mapping $%
\omega :H^{n}\rightarrow H$. This mapping is defined only for tuples $\left(
h_{1},\ldots ,h_{n}\right) \in H^{n}$ such that $h_{k}\in $ $H^{\left(
i_{k}\right) }$, $1\leq k\leq n$. The images of these tuples are elements of
the sort $j$: $\omega \left( h_{1},\ldots ,h_{n}\right) \in H^{\left(
j\right) }$.

In our case the signature $\Omega $ of the representations of the Lie
algebras has this form: 
\begin{equation}
\Omega =\left\{ 0^{\left( 1\right) },-^{\left( 1\right) },\lambda ^{\left(
1\right) }\left( \lambda \in k\right) ,+^{\left( 1\right) },\left[ ,\right]
,0^{\left( 2\right) },-^{\left( 2\right) },\lambda ^{\left( 2\right) }\left(
\lambda \in k\right) ,+^{\left( 2\right) },\circ \right\} .  \label{repAssin}
\end{equation}%
$0^{\left( 2\right) }$ is the $0$-ary operation of taking the zero vector in
the linear space, $\tau _{0^{\left( 2\right) }}=\left( 2\right) $. $%
-^{\left( 2\right) }$ is the unary operation of taking the negative vector
in the linear space, $\tau _{-^{\left( 2\right) }}=\left( 1;2\right) $. $%
+^{\left( 2\right) }$ is the operation of addition of the vectors of the
linear space, $\tau _{+^{\left( 2\right) }}=\left( 2,2;2\right) $. For every 
$\lambda \in k$ we have the unary operation of multiplication of vectors
from the linear space by the scalar $\lambda $. We denote this operation by $%
\lambda $ and $\tau _{\lambda }=\left( 2;2\right) $. $0^{\left( 1\right) }$, 
$-^{\left( 1\right) }$, $\lambda ^{\left( 1\right) }\left( \lambda \in
k\right) $, $+^{\left( 1\right) }$ are the similar operations in the Lie
algebra. $\left[ ,\right] $ is the Lie brackets; this operation has type $%
\tau _{\left[ ,\right] }=\left( 1,1;1\right) $. $\circ $ is an operation of
the action of elements of the Lie algebra on vectors from the linear space, $%
\tau _{\circ }=\left( 1,2;2\right) $.

In universal algebraic geometry we consider some variety $\Theta $ of
universal algebras of the signature $\Omega $. We denote by $%
X_{0}=\bigcup\limits_{i\in \Gamma }X_{0}^{\left( i\right) }$ a set of
symbols, such that $X_{0}^{\left( i\right) }$ is an infinite countable set
for every $i\in \Gamma $ and $X_{0}^{\left( i\right) }\cap X_{0}^{\left(
j\right) }=\varnothing $ when $i\neq j$. By $\mathfrak{F}\left( X_{0}\right) 
$ we denote the set of all finite subsets of $X_{0}$. We will consider the
category $\Theta ^{0}$, whose objects are all free algebras $F\left(
X\right) $ of the variety $\Theta $ generated by finite subsets $X\in 
\mathfrak{F}\left( X_{0}\right) $, such that $\left( F\left( X\right)
\right) ^{\left( i\right) }\supseteq X\cap X_{0}^{\left( i\right) }$.
Morphisms of the category $\Theta ^{0}$ are homomorphisms of these algebras.
We will occasionally denote $F\left( X\right) =F\left( x_{1},x_{2},\ldots
,x_{n}\right) $ if $X=\left\{ x_{1},x_{2},\ldots ,x_{n}\right\} $ and even $%
F\left( X\right) =F\left( x\right) $ if $X$ has only one element.

We consider a "system of equations" $T\subseteq \dbigcup\limits_{i\in \Gamma
}\left( \left( F\right) ^{\left( i\right) }\right) ^{2}$, where $F\in 
\mathrm{Ob}\Theta ^{0}$ (see \cite[Section 4]{TsurkovManySorted}), and we
"resolve" these equations in arbitrary algebra $H\in \Theta $. The set $%
\mathrm{Hom}\left( F,H\right) $ serves as an "affine space over the algebra $%
H$": the solution of the system $T$ is a homomorphism $\mu \in \mathrm{Hom}%
\left( F,H\right) $ such that $\mu \left( t_{1}\right) =\mu \left(
t_{2}\right) $ holds for every $\left( t_{1},t_{2}\right) \in T$ or $%
T\subseteq \ker \mu $. $T_{H}^{\prime }=\left\{ \mu \in \mathrm{Hom}\left(
F,H\right) \mid T\subseteq \ker \mu \right\} $ will be the set of all the
solutions of the system $T$. For every set of "points" $R\subseteq \mathrm{%
Hom}\left( F,H\right) $ we consider a congruence of equations defined in
this way: $R_{H}^{\prime }=\bigcap\limits_{\mu \in R}\ker \mu $. This is a
maximal system of equations which has the set of solutions $R$. For every
set of equations $T$ we consider its algebraic closure $T_{H}^{\prime \prime
}=\bigcap\limits_{\mu \in T_{H}^{\prime }}\ker \mu $ with respect to the
algebra $H$. In the many-sorted case it is possible that $\mathrm{Hom}\left(
F,H\right) =\varnothing $, and in this situation $T_{H}^{\prime \prime
}=\dbigcup\limits_{i\in \Gamma }\left( \left( F\right) ^{\left( i\right)
}\right) ^{2}$ holds for every $T\subseteq \dbigcup\limits_{i\in \Gamma
}\left( \left( F\right) ^{\left( i\right) }\right) ^{2}$. A set $T\subseteq
\dbigcup\limits_{i\in \Gamma }\left( \left( F\right) ^{\left( i\right)
}\right) ^{2}$ is called $H$-closed if $T=T_{H}^{\prime \prime }$. An $H$%
-closed set is always a congruence. We denote the family of all $H$-closed
congruences in $F$ by $Cl_{H}(F)$.

\begin{definition}
Algebras $H_{1},H_{2}\in \Theta $ are \textbf{geometrically equivalent} if
and only if for every $F\in \mathrm{Ob}\Theta ^{0}$ and every $T\subseteq
\dbigcup\limits_{i\in \Gamma }\left( \left( F\right) ^{\left( i\right)
}\right) ^{2}$ the equality $T_{H_{1}}^{\prime \prime }=T_{H_{2}}^{\prime
\prime }$ is fulfilled.
\end{definition}

By this definition, algebras $H_{1},H_{2}\in \Theta $ are geometrically
equivalent if and only if the families $Cl_{H_{1}}(F)$ and $Cl_{H_{2}}(F)$
coincide for every $F\in \mathrm{Ob}\Theta ^{0}$.

\begin{definition}
\label{Autom_equiv}\cite{PlotkinSame}We say that \textit{algebras }$%
H_{1},H_{2}\in \Theta $\textit{\ are \textbf{automorphically equivalent} if
there exist an automorphism }$\Phi :\Theta ^{0}\rightarrow \Theta ^{0}$%
\textit{\ and the bijections}%
\[
\alpha (\Phi )_{F}:Cl_{H_{1}}(F)\rightarrow Cl_{H_{2}}(\Phi (F)) 
\]%
for every $F\in \mathrm{Ob}\Theta ^{0}$, \textit{coordinated in the
following sense: if }$F_{1},F_{2}\in \mathrm{Ob}\Theta ^{0}$\textit{, }$\mu
_{1},\mu _{2}\in \mathrm{Hom}\left( F_{1},F_{2}\right) $\textit{, }$T\in
Cl_{H_{1}}(F_{2})$\textit{\ then}%
\[
\tau \mu _{1}=\tau \mu _{2}, 
\]%
\textit{if and only if }%
\[
\widetilde{\tau }\Phi \left( \mu _{1}\right) =\widetilde{\tau }\Phi \left(
\mu _{2}\right) , 
\]%
\textit{where }$\tau :F_{2}\rightarrow F_{2}/T$\textit{, }$\widetilde{\tau }%
:\Phi \left( F_{2}\right) \rightarrow \Phi \left( F_{2}\right) /\alpha (\Phi
)_{F_{2}}\left( T\right) $\textit{\ are the natural epimorphisms.}
\end{definition}

The definition of the automorphic equivalence in the language the category
of coordinate algebras was considered in \cite{PlotkinSame} and \cite%
{TsurkovManySorted}. Intuitively we can say that algebras $H_{1},H_{2}\in
\Theta $ are automorphically equivalent if and only if the families $%
Cl_{H_{1}}(F)$ and $Cl_{H_{2}}(\Phi \left( F\right) )$ coincide up to a
change of coordinates. This change is defined by the automorphism $\Phi $.

\begin{definition}
\label{inner}An automorphism $\Upsilon $ of an arbitrary category $\mathfrak{%
K}$ is \textbf{inner}, if it is isomorphic as a functor to the identity
automorphism of the category $\mathfrak{K}$.
\end{definition}

It means that for every $F\in \mathrm{Ob}\mathfrak{K}$ there exists an
isomorphism $\sigma _{F}^{\Upsilon }:F\rightarrow \Upsilon \left( F\right) $
such that for every $\mu \in \mathrm{Mor}_{\mathfrak{K}}\left(
F_{1},F_{2}\right) $%
\[
\Upsilon \left( \mu \right) =\sigma _{F_{2}}^{\Upsilon }\mu \left( \sigma
_{F_{1}}^{\Upsilon }\right) ^{-1} 
\]%
\noindent holds. It is clear that the set $\mathfrak{Y}$ of all inner
automorphisms of an arbitrary category $\mathfrak{K}$ is a normal subgroup
of the group $\mathfrak{A}$ of all automorphisms of this category.

By \cite[Proposition 9]{PlotkinSame} and \cite[Theorem 4.2]%
{TsurkovManySorted} (many-sorted case), if an inner automorphism $\Upsilon $
provides the automorphic equivalence of the algebras $H_{1}$ and $H_{2}$,
where $H_{1},H_{2}\in \Theta $, then $H_{1}$ and $H_{2}$ are geometrically
equivalent. Therefore the quotient group $\mathfrak{A/Y}$ measures the
possible difference between the geometric equivalence and automorphic
equivalence of algebras from the variety $\Theta $.

From now on, the word \textquotedblleft representation\textquotedblright\
means a representation of the Lie algebra over the field $k$, which has
characteristic $0$.

We will use the method elaborated on in \cite{PlotkinZhitom} to the
one-sorted algebras and in \cite{TsurkovManySorted} to the many-sorted
algebras for the calculation of the quotient group $\mathfrak{A/Y}$ for the
wide class of varieties of representations. To use this method, we study in
Section \ref{identities} the structure of the free representations in the
varieties of representations. Then we will study in Section \ref{Category}
some properties of the category $\Theta ^{0}$, where $\Theta $ is a variety
of representations.

\section{Homogenization of the identities in the representations of the Lie
algebras\label{identities}}

\setcounter{equation}{0}

In this section we want to clarify and generalize a few the considerations
which can be seen in the beginning of \cite{Simomjan}.

We consider an absolutely free representation $F\left( X\right) $ generated
by the set $X=X^{\left( 1\right) }\cup X^{\left( 2\right) }$, such that $%
\left( F\left( X\right) \right) ^{\left( i\right) }\supseteq X^{\left(
i\right) }$, $i=1,2$.$\ \left( F\left( X\right) \right) ^{\left( 1\right)
}=L\left( X^{\left( 1\right) }\right) $ is a free Lie algebra generated by
the set $X^{\left( 1\right) }$. $\left( F\left( X\right) \right) ^{\left(
2\right) }=A\left( X^{\left( 1\right) }\right) X^{\left( 2\right) }$, where $%
A\left( X^{\left( 1\right) }\right) $ is a free associative algebra with
unit generated by the set $X^{\left( 1\right) }$ and%
\[
A\left( X^{\left( 1\right) }\right) X^{\left( 2\right)
}=\bigoplus\limits_{x^{\left( 2\right) }\in X^{\left( 2\right) }}A\left(
X^{\left( 1\right) }\right) x^{\left( 2\right) }= 
\]%
\[
\mathrm{Sp}_{k}\left\{ x_{i_{n}}^{\left( 1\right) }\ldots x_{i_{1}}^{\left(
1\right) }x^{\left( 2\right) }\mid x_{i_{j}}^{\left( 1\right) }\in X^{\left(
1\right) },x^{\left( 2\right) }\in X^{\left( 2\right) }\right\} 
\]%
is a free left $A\left( X^{\left( 1\right) }\right) $-module generated by
the set $X^{\left( 2\right) }$. For every $l\in \left( F\left( X\right)
\right) ^{\left( 1\right) }=L\left( X^{\left( 1\right) }\right) $ and every $%
v\in \left( F\left( X\right) \right) ^{\left( 2\right) }$ we understand $%
l\circ v$ as $\iota \left( l\right) v$, where $\iota :L\left( X^{\left(
1\right) }\right) \rightarrow A\left( X^{\left( 1\right) }\right) $ is an
embedding, which exists by the Poincar\'{e} - Birkhoff - Witt theorem.

Now we consider an arbitrary subvariety $\Theta $ of the variety of all the
representations. Free representation $F_{\Theta }\left( X\right) $ of $%
\Theta $ generated by the set $X=X^{\left( 1\right) }\cup X^{\left( 2\right)
}$ is the representation $F\left( X\right) /\mathrm{Id}_{\Theta }\left(
X\right) $, where $\mathrm{Id}_{\Theta }\left( X\right) $ is a congruence of
all the identities of the variety $\Theta $ which contain variables from the
set $X$. Actually, the free generators of the representation $F_{\Theta
}\left( X\right) $ have a form $\nu \left( x\right) $, where $\nu :$ $%
F\left( X\right) \rightarrow F\left( X\right) /\mathrm{Id}_{\Theta }\left(
X\right) $ is a natural epimorphism, $X$ is a set of free generators of the
representation $F\left( X\right) $, $x\in X$. But we will use the same
symbols from the free generators of the representations $F\left( X\right) $
and $F_{\Theta }\left( X\right) $. $\left( F_{\Theta }\left( X\right)
\right) ^{\left( 1\right) }=L\left( X^{\left( 1\right) }\right) /I_{\Theta
}\left( X^{\left( 1\right) }\right) $, $\left( F_{\Theta }\left( X\right)
\right) ^{\left( 2\right) }=A\left( X^{\left( 1\right) }\right) X^{\left(
2\right) }/V_{\Theta }\left( X\right) $, where $I_{\Theta }\left( X^{\left(
1\right) }\right) $ is an ideal of the $L\left( X^{\left( 1\right) }\right) $%
, $V_{\Theta }\left( X\right) $ is a $A\left( X^{\left( 1\right) }\right) $%
-left submodule of the $A\left( X^{\left( 1\right) }\right) X^{\left(
2\right) }$ and for every $l\in I_{\Theta }\left( X^{\left( 1\right)
}\right) $ and for every $v\in \left( F_{\Theta }\left( X\right) \right)
^{\left( 2\right) }$ the $l\circ v\in V_{\Theta }\left( X\right) $ holds.
The ideal $I_{\Theta }\left( X^{\left( 1\right) }\right) $ and the submodule 
$V_{\Theta }\left( X\right) $ are fully invariant. The ideal $I_{\Theta
}\left( X^{\left( 1\right) }\right) $ is polyhomogeneous by the Theorem of
the homogenization of the identities of linear algebras (see for example 
\cite[Theorem 4.2.2]{BahturinLieIdent}). So the Lie algebra $L_{\Theta
}\left( X^{\left( 1\right) }\right) =L\left( X^{\left( 1\right) }\right)
/I_{\Theta }\left( X^{\left( 1\right) }\right) $ is a graded algebra.

$\left( F\left( X\right) \right) ^{\left( 2\right) }=\bigoplus\limits_{i\in
I}U_{i}$, where $X^{\left( 2\right) }=\left\{ x_{i}^{\left( 2\right) }\mid
i\in I\right\} $, $U_{i}=A\left( X^{\left( 1\right) }\right) x_{i}^{\left(
2\right) }$ is a free left $A\left( X^{\left( 1\right) }\right) $-cyclic
module. Every element $u\in \left( F\left( X\right) \right) ^{\left(
2\right) }$ has unique decomposition $u=\sum\limits_{i\in I_{u}}u_{i}$,
where $I_{u}\subseteq I$, $\left\vert I_{u}\right\vert <\infty $, $%
u_{i}=f_{i}x_{i}^{\left( 2\right) }\in U_{i}$, $f_{i}\in A\left( X^{\left(
1\right) }\right) $. We shell call the elements $u_{i}$ the \textbf{cyclic
components} of the element $u$.

\begin{proposition}
If $v\in V_{\Theta }\left( X\right) $, then all cyclic components of the $v$
are also elements of $V_{\Theta }\left( X\right) $.
\end{proposition}

\begin{proof}
We will consider a decomposition of $v$ to the cyclic components: $%
v=\sum\limits_{i\in I_{v}}v_{i}=\sum\limits_{i\in I_{v}}f_{i}x_{i}^{\left(
2\right) }$. For every $i\in I_{v}$, there exists $\chi _{i}\in \mathrm{End}%
\left( F\left( X\right) \right) $ such that $\left( \chi _{i}\right) _{\mid
X^{\left( 1\right) }}=id_{X^{\left( 1\right) }}$, $\chi _{i}\left(
x_{i}^{\left( 2\right) }\right) =x_{i}^{\left( 2\right) }$, $\chi _{i}\left(
x_{j}^{\left( 2\right) }\right) =0^{\left( 2\right) }$ for every $j\in
I\smallsetminus \left\{ i\right\} $. $\chi _{i}\left( v\right) =v_{i}\in
V_{\Theta }\left( X\right) $.
\end{proof}

Therefore $V_{\Theta }\left( X\right) =\bigoplus\limits_{i\in I}\left(
U_{i}\cap V_{\Theta }\left( X\right) \right) $ and%
\[
\left( F_{\Theta }\left( X\right) \right) ^{\left( 2\right) }=\left( F\left(
X\right) \right) ^{\left( 2\right) }/V_{\Theta }\left( X\right) \cong
\bigoplus\limits_{i\in I}\left( U_{i}/\left( U_{i}\cap V_{\Theta }\left(
X\right) \right) \right) . 
\]%
$U_{i}\cap V_{\Theta }\left( X\right) =N_{i}=S_{i}x_{i}^{\left( 2\right) }$,
where $S_{i}$ is a left-side ideal of $A\left( X^{\left( 1\right) }\right) $.

\begin{proposition}
$S_{i}$ is a two-sided polyhomogeneous ideal of $A\left( X^{\left( 1\right)
}\right) $.
\end{proposition}

\begin{proof}
If $s\in S_{i}$, then $sx_{i}^{\left( 2\right) }\in V_{\Theta }\left(
X\right) $. We take $f\in A\left( X^{\left( 1\right) }\right) $. There
exists $\chi _{f}\in \mathrm{End}\left( F\left( X\right) \right) $ such that 
$\left( \chi _{f}\right) _{\mid X^{\left( 1\right) }}=id_{X^{\left( 1\right)
}}$, $\chi _{f}\left( x_{i}^{\left( 2\right) }\right) =fx_{i}^{\left(
2\right) }$, $\chi _{f}\left( x_{j}^{\left( 2\right) }\right) =x_{j}^{\left(
2\right) }$ for every $j\in I\smallsetminus \left\{ i\right\} $. $\chi
_{f}\left( sx_{i}^{\left( 2\right) }\right) =sfx_{i}^{\left( 2\right) }\in
V_{\Theta }\left( X\right) $, so $sf\in S_{i}$.

The proof of the fact that $S_{i}$ is a polyhomogeneous ideal is a very
similar to the proof of the theorem of the homogenization of the identities
of linear algebras.
\end{proof}

Therefore $U_{i}/\left( U_{i}\cap V_{\Theta }\left( X\right) \right)
=A\left( X^{\left( 1\right) }\right) x_{i}^{\left( 2\right)
}/S_{i}x_{i}^{\left( 2\right) }\cong \left( A\left( X^{\left( 1\right)
}\right) /S_{i}\right) x_{i}^{\left( 2\right) }$, where $A\left( X^{\left(
1\right) }\right) /S_{i}=A_{i}$ is a graded algebra.

\begin{proposition}
There exists a two-sided polyhomogeneous ideal $S_{\Theta }\left( X^{\left(
1\right) }\right) \leq A\left( X^{\left( 1\right) }\right) $, such that $%
V_{\Theta }\left( X\right) =S_{\Theta }\left( X^{\left( 1\right) }\right)
X^{\left( 2\right) }=\bigoplus\limits_{i\in I}S_{\Theta }\left( X^{\left(
1\right) }\right) x_{i}^{\left( 2\right) }$.
\end{proposition}

\begin{proof}
We only need to prove that $S_{i}=S_{j}$ for every $i,j\in I$. We consider $%
s\in S_{i}$. $sx_{i}^{\left( 2\right) }\in U_{i}\cap V_{\Theta }\left(
X\right) $. There exists $\chi _{j}\in \mathrm{End}\left( F\left( X\right)
\right) $ such that $\left( \chi _{j}\right) _{\mid X^{\left( 1\right)
}}=id_{X^{\left( 1\right) }}$, $\chi _{j}\left( x_{i}^{\left( 2\right)
}\right) =x_{j}^{\left( 2\right) }$, $\chi _{j}\left( x_{k}^{\left( 2\right)
}\right) =0^{\left( 2\right) }$ for every $k\in I\smallsetminus \left\{
i\right\} $. $\chi _{j}\left( sx_{i}^{\left( 2\right) }\right)
=sx_{j}^{\left( 2\right) }\in U_{j}\cap V_{\Theta }\left( X\right) $.
Therefore $s\in S_{j}$.
\end{proof}

Therefore we prove the following:

\begin{theorem}
\label{homogident}$\left( F_{\Theta }\left( X\right) \right) ^{\left(
1\right) }=L_{\Theta }\left( X^{\left( 1\right) }\right) =$ $L\left(
X^{\left( 1\right) }\right) /I_{\Theta }\left( X^{\left( 1\right) }\right) $%
, $\left( F_{\Theta }\left( X\right) \right) ^{\left( 2\right)
}=\bigoplus\limits_{x\in X^{\left( 2\right) }}\left( A_{\Theta }\left(
X^{\left( 1\right) }\right) x\right) $, where $A_{\Theta }\left( X^{\left(
1\right) }\right) =A\left( X^{\left( 1\right) }\right) /S_{\Theta }\left(
X^{\left( 1\right) }\right) $, $L\left( X^{\left( 1\right) }\right) $ is a
free Lie algebra, generated by the set $X^{\left( 1\right) }$, $I_{\Theta
}\left( X^{\left( 1\right) }\right) $ is a polyhomogeneous ideal of this
algebra, $A\left( X^{\left( 1\right) }\right) $ is a free associative
algebra with unit, generated by the set $X^{\left( 1\right) }$, $S_{\Theta
}\left( X^{\left( 1\right) }\right) $ is a polyhomogeneous two-sided ideal
of this algebra, $A_{\Theta }\left( X^{\left( 1\right) }\right) =A\left(
X^{\left( 1\right) }\right) /S_{\Theta }\left( X^{\left( 1\right) }\right) $
is a graded algebra. Also for every $a\in A\left( X^{\left( 1\right)
}\right) $ and every $l\in I_{\Theta }\left( X^{\left( 1\right) }\right) $
is fulfilled $al,la\in S_{\Theta }\left( X^{\left( 1\right) }\right) $ ($%
I_{\Theta }\left( X^{\left( 1\right) }\right) \subseteq \left( S_{\Theta
}\left( X^{\left( 1\right) }\right) :A\left( X^{\left( 1\right) }\right)
\right) $).
\end{theorem}

\section{Category of the finitely generated free representations\label%
{Category}}

\setcounter{equation}{0}

We consider the category $\Theta ^{0}$, where $\Theta $ is an arbitrary
subvariety of the variety of all the representations.

\begin{definition}
\label{IBN}We say that the variety $\Theta $ has an \textbf{IBN propriety}
if for every $F_{\Theta }\left( X\right) ,F_{\Theta }\left( Y\right) \in 
\mathrm{Ob}\Theta ^{0}$ the $F_{\Theta }\left( X\right) \cong F_{\Theta
}\left( Y\right) $ holds if and only if $\left\vert X^{\left( i\right)
}\right\vert =\left\vert Y^{\left( i\right) }\right\vert $, $i=1,2$.
\end{definition}

We consider the nontrivial variety $\Theta $. It means that in the variety $%
\Theta $ the identity $x^{\left( 2\right) }=0$ is not fulfilled.

\begin{proposition}
Every nontrivial variety $\Theta $ has IBN propriety.
\end{proposition}

\begin{proof}
We consider $F_{\Theta }\left( X\right) =F_{\Theta }\in \mathrm{Ob}\Theta
^{0}$. We will denote $A\left( X^{\left( 1\right) }\right) =A$, $F_{\Theta
}^{\left( 1\right) }=L_{\Theta }\left( X^{\left( 1\right) }\right)
=L_{\Theta }$, $A_{\Theta }\left( X^{\left( 1\right) }\right) =A_{\Theta }$, 
$S_{\Theta }\left( X^{\left( 1\right) }\right) =S_{\Theta }$. We denote by $%
J $ the two-sided ideal of $A$ generated by set $X^{\left( 1\right) }$: $%
J=\left\langle X^{\left( 1\right) }\right\rangle _{idealA}$.

By \cite[Section 3]{TsurkovClassicalVar} we have%
\[
\left\vert X^{\left( 1\right) }\right\vert =\dim _{k}\left( L_{\Theta }/ 
\left[ L_{\Theta },L_{\Theta }\right] \right) . 
\]

We will use the description of $F_{\Theta }\left( X\right) $ given in
Theorem \ref{homogident}. $\left( F_{\Theta }^{\left( 1\right) }\right)
\circ \left( F_{\Theta }^{\left( 2\right) }\right) =L_{\Theta }\circ \left(
\bigoplus\limits_{x\in X^{\left( 2\right) }}A_{\Theta }x\right)
=\bigoplus\limits_{x\in X^{\left( 2\right) }}\left( L_{\Theta }\cdot
A_{\Theta }\right) x$. $L_{\Theta }\cdot A_{\Theta }$ is a two-sided ideal
of $A_{\Theta }$, because $X^{\left( 1\right) }\subset L_{\Theta }$. $\Theta 
$ is the nontrivial variety, so $J\supseteq S_{\Theta }$ and $L_{\Theta
}\cdot A_{\Theta }=J/S_{\Theta }$. $A_{\Theta }/\left( L_{\Theta }\cdot
A_{\Theta }\right) =\left( A/S_{\Theta }\right) /\left( J/S_{\Theta }\right)
\cong A/J\cong k$.%
\[
F_{\Theta }^{\left( 2\right) }/\left( F_{\Theta }^{\left( 1\right) }\circ
F_{\Theta }^{\left( 2\right) }\right) =\left( \bigoplus\limits_{x\in
X^{\left( 2\right) }}A_{\Theta }x\right) /\left( \bigoplus\limits_{x\in
X^{\left( 2\right) }}\left( L_{\Theta }\cdot A_{\Theta }\right) x\right)
\cong 
\]%
\[
\bigoplus\limits_{x\in X^{\left( 2\right) }}\left( A_{\Theta }/\left(
L_{\Theta }\cdot A_{\Theta }\right) \right) x, 
\]%
hence%
\[
\dim _{k}\left( F_{\Theta }^{\left( 2\right) }/\left( F_{\Theta }^{\left(
1\right) }\circ F_{\Theta }^{\left( 2\right) }\right) \right) =\left\vert
X^{\left( 2\right) }\right\vert \text{.} 
\]
\end{proof}

We say that the variety $\Theta $ is an action-type variety, if $I_{\Theta
}\left( X^{\left( 1\right) }\right) =\left\{ 0\right\} $ for every $%
X^{\left( 1\right) }\subset X_{0}^{\left( 1\right) }$, such that $\left\vert
X^{\left( 1\right) }\right\vert <\infty $.

\begin{proposition}
\label{propMonoiso}If $\Theta $ is an action-type nontrivial variety, then $%
\Phi \left( F_{\Theta }\left( x^{\left( 1\right) }\right) \right) =F_{\Theta
}\left( x^{\left( 1\right) }\right) $ and $\Phi \left( F_{\Theta }\left(
x^{\left( 2\right) }\right) \right) =F_{\Theta }\left( x^{\left( 2\right)
}\right) $ hold for every $\Phi \in \mathrm{Aut}\Theta ^{0}$.
\end{proposition}

\begin{proof}
$\Theta $ has IBN propriety, so, by \cite[Proposition 5.2]{ShestTsur}\ and 
\cite[Section 5]{TsurkovManySorted}, we have two possibilities for every $%
\Phi \in \mathrm{Aut}\Theta ^{0}$: or $\Phi \left( F_{\Theta }\left(
x^{\left( 1\right) }\right) \right) =F_{\Theta }\left( x^{\left( 1\right)
}\right) $ and $\Phi \left( F_{\Theta }\left( x^{\left( 2\right) }\right)
\right) =F_{\Theta }\left( x^{\left( 2\right) }\right) $, or $\Phi \left(
F_{\Theta }\left( x^{\left( 1\right) }\right) \right) =F_{\Theta }\left(
x^{\left( 2\right) }\right) $ and $\Phi \left( F_{\Theta }\left( x^{\left(
2\right) }\right) \right) =F_{\Theta }\left( x^{\left( 1\right) }\right) $. $%
\left( F_{\Theta }\left( x_{1}^{\left( 1\right) },\ldots ,x_{n}^{\left(
1\right) }\right) \right) ^{\left( 1\right) }=L\left( x_{1}^{\left( 1\right)
},\ldots ,x_{n}^{\left( 1\right) }\right) $, $\left( F_{\Theta }\left(
x_{1}^{\left( 1\right) },\ldots ,x_{n}^{\left( 1\right) }\right) \right)
^{\left( 2\right) }=\left\{ 0\right\} $; $\left( F_{\Theta }\left(
x_{1}^{\left( 2\right) },\ldots ,x_{n}^{\left( 2\right) }\right) \right)
^{\left( 1\right) }=\left\{ 0\right\} $, $\left( F_{\Theta }\left(
x_{1}^{\left( 2\right) },\ldots ,x_{n}^{\left( 2\right) }\right) \right)
^{\left( 2\right) }=\mathrm{Sp}_{k}\left( x_{1}^{\left( 2\right) },\ldots
,x_{n}^{\left( 2\right) }\right) $. Now we can use the argument of \cite[%
Proposition 5.9]{TsurkovManySorted} and conclude that there does not exist $%
\Phi \in \mathrm{Aut}\Theta ^{0}$ such that $\Phi \left( F_{\Theta }\left(
x^{\left( 1\right) }\right) \right) =F_{\Theta }\left( x^{\left( 2\right)
}\right) $ and $\Phi \left( F_{\Theta }\left( x^{\left( 2\right) }\right)
\right) =F_{\Theta }\left( x^{\left( 1\right) }\right) $.
\end{proof}

\section{Method of the verbal operations}

\setcounter{equation}{0}

In the beginning of this section we will explain the method of \cite%
{PlotkinZhitom} and \cite{TsurkovManySorted} in the case of arbitrary
variety $\Theta $ of universal algebras of the signature $\Omega $.

In \cite{PlotkinZhitom} the notion of the strongly stable automorphism of
the category $\Theta ^{0}$ was defined. In the case of the variety of
many-sorted algebras ($\left\vert \Gamma \right\vert >1$) we have the
following:

\begin{definition}
\label{str_stab_aut}\cite{TsurkovManySorted}\textit{An automorphism $\Phi $
of the category }$\Theta ^{0}$\textit{\ is called \textbf{strongly stable}
if it satisfies the conditions:}

\begin{enumerate}
\item[1] $\Phi $\textit{\ preserves all objects of }$\Theta ^{0}$\textit{,}

\item[2] \textit{there exists a system of bijections }$S=\left\{
s_{F}:F\rightarrow F\mid F\in \mathrm{Ob}\Theta ^{0}\right\} $\textit{\ such
that all these bijections }conform with the sorting:%
\[
\eta _{F}=\eta _{F}s_{F} 
\]

\item[3] $\Phi $\textit{\ acts on the morphisms }$\mu \in \mathrm{Mor}%
_{\Theta ^{0}}\left( F_{1},F_{2}\right) $\textit{\ of }$\Theta ^{0}$\textit{%
\ thusly: }%
\[
\Phi \left( \mu \right) =s_{F_{2}}\mu s_{F_{1}}^{-1}, 
\]

\item[4] $s_{F}\mid _{X}=id_{X},$ \textit{\ for every }$F\left( X\right) \in 
\mathrm{Ob}\Theta ^{0}$.
\end{enumerate}
\end{definition}

It is clear that the set $\mathfrak{S}$ of all strongly stable automorphisms
of the category $\Theta ^{0}$ is a subgroup of the group $\mathfrak{A}$ of
all automorphisms of this category. By \cite[Theorem 2.3]{TsurkovManySorted}%
, $\mathfrak{A=YS}$ holds if in the category $\Theta ^{0}$ the

\begin{condition}
\label{monoiso}$\Phi \left( F\left( x^{\left( i\right) }\right) \right)
\cong F\left( x^{\left( i\right) }\right) $ for every automorphism $\Phi $
of the category $\Theta ^{0}$, every sort $i\in \Gamma $ and every $%
x^{\left( i\right) }\in X_{0}^{\left( i\right) }\subset X_{0}$
\end{condition}

holds. In this case we have that $\mathfrak{A/Y\cong S/S\cap Y}$. So we must
compute the groups $\mathfrak{S}$ and $\mathfrak{S\cap Y}$.

The group $\mathfrak{S}$ we can compute by the method of verbal operations.
For every word $w=w\left( x_{1},\ldots ,x_{n}\right) \in F\left(
x_{1},\ldots ,x_{n}\right) =F\in \mathrm{Ob}\Theta ^{0}$ and every algebra $%
H\in \Theta $ we can define an operation $w_{H}^{\ast }$ in $H$: if $%
h_{1},\ldots ,h_{n}\in H$ such that $\eta _{H}\left( h_{i}\right) =\eta
_{F}\left( x_{i}\right) $, where $1\leq i\leq n$, then $w_{H}^{\ast }\left(
h_{1},\ldots ,h_{n}\right) =\alpha \left( w\right) $, where $\alpha
:F\rightarrow H$ is a homomorphism such that $\alpha \left( x_{i}\right)
=h_{i}$. If $\eta _{F}\left\{ x_{1},\ldots ,x_{n}\right\} \nsubseteq \Gamma
_{H}$ then the operation $w_{H}^{\ast }$ is defined on the empty subset of $%
H^{n}$. The operation $w_{H}^{\ast }$ is called the verbal operation defined
by the word $w$. This operation we consider as the operation of the type $%
\left( \eta _{F}\left( x_{1}\right) ,\ldots ,\eta _{F}\left( x_{n}\right)
;\eta _{F}\left( w\right) \right) $ even if not all of the free generators $%
x_{1},\ldots ,x_{n}$ really enter into the word $w$.

If we have a system of words $W=\left\{ w_{i}\mid i\in I\right\} $ such that 
$w_{i}\in F_{i}\in \mathrm{Ob}\Theta ^{0}$, then for every $H\in \Theta $ we
denote by $H_{W}^{\ast }$ the universal algebra which coincides with $H$ as
a set with "sorting", but has only verbal operations defined by the words
from $W$.

For the operation $\omega \in \Omega $ which has a type $\tau _{\omega
}=\left( i_{1},\ldots ,i_{n};j\right) $, we take $F_{\omega }=F\left(
X_{\omega }\right) \in \mathrm{Ob}\Theta ^{0}$ such that $X_{\omega
}=\left\{ x^{\left( i_{1}\right) },\ldots ,x^{\left( i_{n}\right) }\right\} $%
, $\eta _{A_{\omega }}\left( x^{\left( i_{k}\right) }\right) =i_{k}$, $1\leq
k\leq n$. By the method of the verbal operations (see \cite{PlotkinZhitom}, 
\cite{TsurAutomEqAlg} and \cite{TsurkovManySorted}) there is a bijection
between the set of the strongly stable automorphisms of the category $\Theta
^{0}$ and the set of the systems of words $W$ which fulfill the

\begin{condition}
\label{operationconditions}

\begin{enumerate}
\item $W=\left\{ w_{\omega }\in F_{\omega }\mid \omega \in \Omega \right\} $,

\item for every $F\left( X\right) \in \mathrm{Ob}\Theta ^{0}$ there exists a
bijection $s_{F}:F\rightarrow F$ such that $\left( s_{F}\right) _{\mid
X}=id_{X}$ and $s_{F}:F\rightarrow F_{W}^{\ast }$ is an isomorphism.
\end{enumerate}
\end{condition}

We can compute the group $\mathfrak{S\cap Y}$ by this

\begin{criterion}
\label{intersectioncriterion}\cite[Proposition 3.7]{TsurkovManySorted}The
strongly stable automorphism $\Phi $ which corresponds to the system of
words $W^{\Phi }=W$ is inner if and only if there is a system of
isomorphisms $\left\{ \tau _{F}:F\rightarrow F_{W}^{\ast }\mid F\in \mathrm{%
Ob}\Theta ^{0}\right\} $ such that for every $F_{1},F_{2}\in \mathrm{Ob}%
\Theta ^{0}$ and every $\mu \in \mathrm{Mor}_{\Theta ^{0}}\left(
F_{1},F_{2}\right) $ the 
\[
\tau _{F_{2}}\mu =\mu \tau _{F_{1}} 
\]%
holds.
\end{criterion}

Now we came back to the varieties of the representations. By Proposition \ref%
{propMonoiso} we have that in an action-type nontrivial variety of the
representations, Condition \ref{monoiso} is fulfilled. The signature $\Omega 
$ of the representations was described in (\ref{repAssin}). For every $%
\omega \in \Omega $ we must find all the possible forms of the words $%
w_{\omega }$ such that the system $W=\left\{ w_{\omega }\mid \omega \in
\Omega \right\} $ fulfills Condition \ref{operationconditions}.

We say that the variety of the representations is degenerated if the
identity $\left[ x_{1}^{\left( 1\right) },x_{2}^{\left( 1\right) }\right]
\circ x^{\left( 2\right) }=0$ is fulfilled in this variety. It is clear that
the nondegenerated variety is nontrivial.

\begin{proposition}
If $\Theta $ is an action-type and nondegenerated variety of
representations, then the system $W=\left\{ w_{\omega }\mid \omega \in
\Omega \right\} $, which fulfills Condition \ref{operationconditions} has a
form%
\[
w_{0^{\left( 1\right) }}=0^{\left( 1\right) },w_{-^{\left( 1\right)
}}=-x^{\left( 1\right) },w_{\lambda ^{\left( 1\right) }}=\varphi \left(
\lambda \right) x^{\left( 1\right) }, 
\]%
\begin{equation}
w_{+^{\left( 1\right) }}=x_{1}^{\left( 1\right) }+x_{2}^{\left( 1\right)
},w_{\left[ ,\right] }=a\left[ x_{1}^{\left( 1\right) },x_{2}^{\left(
1\right) }\right] ,  \label{possibles_words}
\end{equation}%
\[
w_{0^{\left( 2\right) }}=0^{\left( 2\right) },w_{-^{\left( 2\right)
}}=-x^{\left( 2\right) },w_{\lambda ^{\left( 2\right) }}=\varphi \left(
\lambda \right) x^{\left( 2\right) }, 
\]%
\[
w_{+^{\left( 2\right) }}=x_{1}^{\left( 2\right) }+x_{2}^{\left( 2\right)
},w_{\circ }=a\left( x^{\left( 1\right) }\circ x^{\left( 2\right) }\right) , 
\]%
where $\varphi \in \mathrm{Aut}k$, $a\in k^{\ast }$.{}
\end{proposition}

\begin{proof}
By computations in $F_{\Theta }\left( \varnothing \right) $, $F_{\Theta
}\left( x^{\left( 1\right) }\right) $ and $F_{\Theta }\left( x_{1}^{\left(
1\right) },x_{2}^{\left( 1\right) }\right) $ we can conclude that $%
w_{0^{\left( 1\right) }}=0^{\left( 1\right) }$, $w_{-^{\left( 1\right)
}}=-x^{\left( 1\right) }$, $w_{\lambda ^{\left( 1\right) }}=\varphi \left(
\lambda \right) x^{\left( 1\right) }$, $w_{+^{\left( 1\right)
}}=x_{1}^{\left( 1\right) }+x_{2}^{\left( 1\right) }$, $w_{\left[ ,\right]
}=a\left[ x_{1}^{\left( 1\right) },x_{2}^{\left( 1\right) }\right] $, where $%
\varphi \in \mathrm{Aut}k$, $a\in k^{\ast }$. These computations can be seen
in \cite{PlotkinZhitom} and \cite[Section 4]{TsurkovClassicalVar}. By the
computations in $F_{\Theta }\left( \varnothing \right) $, $F_{\Theta }\left(
x^{\left( 2\right) }\right) $ and $F_{\Theta }\left( x_{1}^{\left( 2\right)
},x_{2}^{\left( 2\right) }\right) $ we can conclude that $w_{0^{\left(
2\right) }}=0^{\left( 2\right) }$, $w_{-^{\left( 2\right) }}=-x^{\left(
2\right) }$, $w_{\lambda ^{\left( 2\right) }}=\psi \left( \lambda \right)
x^{\left( 2\right) }$, $w_{+^{\left( 2\right) }}=x_{1}^{\left( 2\right)
}+x_{2}^{\left( 2\right) }$, where $\psi \in \mathrm{Aut}k$. These
computations are very simple. $w_{\circ }$ must be calculated in $F_{\Theta
}\left( x^{\left( 1\right) },x^{\left( 2\right) }\right) $. $\Theta $ is a
nondegenerated variety, so the identity $x^{\left( 1\right) }\circ x^{\left(
2\right) }=0$ is not fulfilled in this variety. Therefore, by Theorem \ref%
{homogident} $A_{\Theta }\left( x^{\left( 1\right) }\right) =A\left(
x^{\left( 1\right) }\right) /S_{\Theta }\left( x^{\left( 1\right) }\right) =k%
\left[ x^{\left( 1\right) }\right] /\left( \left( x^{\left( 1\right)
}\right) ^{d}\right) $, where $d>1$, and $w_{\circ }\left( x_{1}^{\left(
2\right) },x_{2}^{\left( 2\right) }\right) =f\left( x^{\left( 1\right)
}\right) \circ x^{\left( 2\right) }$, where $f\left( x^{\left( 1\right)
}\right) \in k\left[ x^{\left( 1\right) }\right] $, $\deg f\left( x^{\left(
1\right) }\right) <d$.

We use the method of the \cite[Theorem 5.4]{TsurkovManySorted}. In $%
F_{\Theta }\left( x^{\left( 1\right) },x^{\left( 2\right) }\right) $ the
equality $\lambda \left( x^{\left( 1\right) }\circ x^{\left( 2\right)
}\right) =\left( \lambda x^{\left( 1\right) }\right) \circ x^{\left(
2\right) }$ holds for every $\lambda \in k$. If there exists a bijection $%
s_{F}:F\rightarrow F$, where $F=F_{\Theta }\left( x^{\left( 1\right)
},x^{\left( 2\right) }\right) $, such that $\left( s_{F}\right) _{\mid
X}=id_{X}$, where $X=\left\{ x^{\left( 1\right) },x^{\left( 2\right)
}\right\} $, and $s_{F}:F\rightarrow F_{W}^{\ast }$ is an isomorphism, then $%
s_{F}\left( \lambda \left( x^{\left( 1\right) }\circ x^{\left( 2\right)
}\right) \right) =w_{\lambda ^{\left( 2\right) }}\left( w_{\circ }\left(
x^{\left( 1\right) },x^{\left( 2\right) }\right) \right) $ and $s_{F}\left(
\left( \lambda x^{\left( 1\right) }\right) \circ x^{\left( 2\right) }\right)
=w_{\circ }\left( w_{\lambda ^{\left( 1\right) }}\left( x^{\left( 1\right)
}\right) ,x^{\left( 2\right) }\right) $. Now we can conclude that $w_{\circ
}\left( x_{1}^{\left( 2\right) },x_{2}^{\left( 2\right) }\right) =bx^{\left(
1\right) }\circ x^{\left( 2\right) }$, where $b\in k\setminus \left\{
0\right\} $.

We consider the nondegenerated variety. Thus, by the same method, we
conclude that $a=b$ from the equality $\left[ x_{1}^{\left( 1\right)
},x_{2}^{\left( 1\right) }\right] \circ x^{\left( 2\right) }=x_{1}^{\left(
1\right) }\circ \left( x_{2}^{\left( 1\right) }\circ x^{\left( 2\right)
}\right) -x_{2}^{\left( 1\right) }\circ \left( x_{1}^{\left( 1\right) }\circ
x^{\left( 2\right) }\right) $ in $F_{\Theta }\left( x_{1}^{\left( 1\right)
},x_{2}^{\left( 1\right) },x^{\left( 2\right) }\right) $.

Also we conclude $\varphi =\psi $ from the equality $\left( \lambda
x^{\left( 1\right) }\right) \circ x^{\left( 2\right) }=x^{\left( 1\right)
}\circ \left( \lambda x^{\left( 2\right) }\right) $ which holds in $%
F_{\Theta }\left( x^{\left( 1\right) },x^{\left( 2\right) }\right) $ for
every $\lambda \in k$.
\end{proof}

\begin{theorem}
\label{mainTheorem}If $\Theta $ is an action-type and nondegenerated variety
of representations and this variety is defined by identities with
coefficients from $%
\mathbb{Z}
$, then $\mathfrak{A/Y\cong }\mathrm{Aut}k$.
\end{theorem}

\begin{proof}
Now we will prove that Condition \ref{operationconditions} is fulfilled for
the systems of words $W$ which have a form (\ref{possibles_words}). We will
consider an absolutely free representation $F\left( X\right) =F$ generated
by the set $X=X^{\left( 1\right) }\cup X^{\left( 2\right) }$, such that $%
F^{\left( i\right) }\supset X^{\left( i\right) }$, $i=1,2$, where $X\in 
\mathfrak{F}\left( X_{0}\right) $. By \cite[Theorem 5.4]{TsurkovManySorted},
there exists an isomorphism $\widetilde{s}_{F}:F\rightarrow F_{W}^{\ast }$,
such that $\left( \widetilde{s}_{F}\right) _{\mid X}=id_{X}$. There exists
the natural epimorphism $\nu :F\left( X\right) \rightarrow F_{\Theta }\left(
X\right) =F\left( X\right) /\mathrm{Id}_{\Theta }\left( X\right) $, where $%
F_{\Theta }\left( X\right) =F_{\Theta }$ is a free representation of the
variety $\Theta $ generated by the set $X$. In truth, $F_{\Theta }\left(
X\right) $ is generated by the set $\nu \left( X\right) $ but we will use
our short notation. The operations defined in the representations $%
F_{W}^{\ast }$ and $\left( F_{\Theta }\right) _{W}^{\ast }$ are the verbal
operations, so $\nu :F_{W}^{\ast }\rightarrow \left( F_{\Theta }\right)
_{W}^{\ast }$ is also an epimorphism. By Theorem \ref{homogident},%
\[
\left( F_{\Theta }\left( X\right) \right) ^{\left( 1\right) }=L\left(
X^{\left( 1\right) }\right) ,\left( F_{\Theta }\left( X\right) \right)
^{\left( 2\right) }=\bigoplus\limits_{x\in X^{\left( 2\right) }}\left(
\left( A\left( X^{\left( 1\right) }\right) /S_{\Theta }\left( X^{\left(
1\right) }\right) \right) x\right) , 
\]%
where $S_{\Theta }\left( X^{\left( 1\right) }\right) $ is a polyhomogeneous
two-sided ideal of $A\left( X^{\left( 1\right) }\right) $. So, the variety $%
\Theta $ can be defined by the identities, which have a form $fx^{\left(
2\right) }=0$, where $f$ is a polyhomogeneous element of the ideal $%
S_{\Theta }\left( X^{\left( 1\right) }\right) $. $f=\sum\limits_{i\in
I}\lambda _{i}m_{i}$, where $\left\vert I\right\vert <\infty $, $m_{i}$ are
monomials of $A\left( X^{\left( 1\right) }\right) $. By our assumption, we
can suppose that $\lambda _{i}\in 
\mathbb{Z}
$. So, as in \cite[proof of the Theorem 4.1]{TsurkovClassicalVar}, $%
\widetilde{s}_{F}\left( fx^{\left( 2\right) }\right) =\sum\limits_{i\in
I}\varphi \left( \lambda _{i}\right) a^{\deg f}m_{i}x^{\left( 2\right)
}=a^{\deg f}\sum\limits_{i\in I}\lambda _{i}m_{i}x^{\left( 2\right)
}=a^{\deg f}fx^{\left( 2\right) }\in S_{\Theta }\left( X^{\left( 1\right)
}\right) x^{\left( 2\right) }$. Hence, there exists a homomorphism $%
s_{F_{\Theta }}:F_{\Theta }\rightarrow \left( F_{\Theta }\right) _{W}^{\ast
} $, such that $s_{F_{\Theta }}\nu =\nu \widetilde{s}_{F}$. In particular,
if $x\in X$, then $s_{F_{\Theta }}\nu \left( x\right) =\nu \left( x\right) $%
, or, in our short notation, $\left( s_{F_{\Theta }}\right) _{\mid X}=id_{X}$%
.

As in \cite[proof of Theorem 5.4]{TsurkovManySorted}, we can prove, at
first, that $s_{F_{\Theta }}$ is an isomorphism, and then, by Criterion \ref%
{intersectioncriterion}, that the strongly stable automorphism which
corresponds to the system of words (\ref{possibles_words}) is inner if and
only if $\varphi =id_{k}$. Therefore $\mathfrak{A/Y\cong }\mathrm{Aut}k$.
\end{proof}

\section{Example\label{example}}

\setcounter{equation}{0}

In this section we consider a field $k$ such that $\mathrm{Aut}k\neq \left\{
id_{k}\right\} $. We will give an example of the variety $\Theta $ of
representations, which fulfills the conditions of the Theorem \ref%
{mainTheorem}, and two representations $H_{1},H_{2}\in \Theta $, such that
they are automorphically equivalent, but not geometrically equivalent. This
example is similar to the examples of \cite[Example 3]{TsurkovClassicalVar}
and \cite[Subsection 5.4]{TsurkovManySorted}.

We consider the variety $\Theta $ of representations, defined by identity%
\[
x_{1}^{\left( 1\right) }x_{2}^{\left( 1\right) }\ldots x_{5}^{\left(
1\right) }x_{6}^{\left( 1\right) }x^{\left( 2\right) }=0. 
\]%
In this variety we consider the free algebra $F=F_{\Theta }\left(
x_{1}^{\left( 1\right) },x_{2}^{\left( 1\right) },x^{\left( 2\right)
}\right) $.\linebreak $A_{\Theta }\left( x_{1}^{\left( 1\right)
},x_{2}^{\left( 1\right) }\right) $ contains two linear independent elements%
\[
\iota \left[ x_{1}^{\left( 1\right) },\left[ x_{1}^{\left( 1\right) },\left[ %
\left[ x_{1}^{\left( 1\right) },x_{2}^{\left( 1\right) }\right]
,x_{2}^{\left( 1\right) }\right] \right] \right] =e_{1} 
\]%
and%
\[
\iota \left[ \left[ x_{1}^{\left( 1\right) },\left[ x_{1}^{\left( 1\right)
},x_{2}^{\left( 1\right) }\right] \right] ,\left[ x_{1}^{\left( 1\right)
},x_{2}^{\left( 1\right) }\right] \right] =e_{2}. 
\]%
We denote by $W$ the system of words (\ref{possibles_words}), such that $a=1$%
, $\varphi \neq \left\{ id_{k}\right\} $, and by $\Phi $ the strongly stable
automorphism of the category $\Theta ^{0}$, defined by the system of words $%
W $. There exists $\lambda \in k$ such that $\varphi \left( \lambda \right)
\neq \lambda $. We denote $t=\lambda e_{1}+e_{2}$. We denote by $T$ the
two-sided ideal of the algebra $A_{\Theta }\left( x_{1}^{\left( 1\right)
},x_{2}^{\left( 1\right) }\right) $ generated by the element $t$. In truth, $%
T=\mathrm{sp}_{k}\left( t\right) $. We denote by $H$ the representation such
that $\left( H\right) ^{\left( 1\right) }=\left( F\right) ^{\left( 1\right)
} $, $\left( H\right) ^{\left( 2\right) }=\left( A_{\Theta }\left(
x_{1}^{\left( 1\right) },x_{2}^{\left( 1\right) }\right) /T\right) x^{\left(
2\right) }$. $H\in \Theta $ by the Birkhoff Theorem. $H_{W}^{\ast }\in
\Theta $ by \cite[Proposition 3.5]{TsurkovManySorted}. By \cite[Corollary 1
from the Proposition 4.2]{TsurkovManySorted}, the representations $H$ and $%
H_{W}^{\ast }$ are automorphically equivalent. By computations which are
very similar to the computations of \cite[Example 3]{TsurkovClassicalVar}
and \cite[Subsection 5.4]{TsurkovManySorted}, $H$ and $H_{W}^{\ast }$ are
not geometrically equivalent.

\section{Acknowledgements}

I would like to acknowledge the support of PNPD CAPES -- Programa Nacional
de P\'{o}s-Doutorado da Coordena\c{c}\~{a}o de Aperfei\c{c}oamento de
Pessoal de N\'{\i}vel Superior (National Postdoctoral Program of the
Coordination for the Improvement of Higher Education Personnel, Brazil) and
of CNPq - Conselho Nacional de Desenvolvimento Cient\'{\i}fico e Tecnol\'{o}%
gico (National Council for Scientific and Technological Development,
Brazil), Project 314045/2013-9, for bestowing a visiting researcher
scholarship.

I also would like to acknowledge Professors Nir Cohen and David Armando
Zavaleta Villanueva, who initiated my participation in these two programs.

\end{document}